\documentclass[12pt, preprint, authoryear]{elsarticle}

\usepackage{amsmath,amsfonts,amssymb,amsthm}
\usepackage{color,graphicx}
\usepackage{epsfig,fullpage}
\usepackage{natbib,cite}
\usepackage{url}
\usepackage{lineno}
\usepackage{subfigure}
\usepackage{mathtools,MnSymbol}
\usepackage{soul}
\usepackage{enumitem}
\usepackage{hyperref}
\usepackage[a4paper, margin = 1in]{geometry}
\usepackage{setspace}

\newtheorem{theorem}{Theorem}[section]

\newtheorem{lemma}[theorem]{Lemma}
\newtheorem{proposition}[theorem]{Proposition}

\newcommand{\msum}{\mathop{\sum}}

\newcommand{\xbar}{\bar{X}_n}
\newcommand{\twofam}[2]{#1^{(#2)}}
\newcommand{\mtxt}[1]{\mathcal{#1}_n}
\newcommand{\varmn}{{\rm var} \left(\mtxt{M}\right)}

\newcommand{\transpose}[1]{#1^{\top}}
\newcommand{\sammean}{\transpose{\bar{X}_n} \bar{X}_n}
\newcommand{\tr}{{\rm tr}}

\newcommand{\cp}{\overset{p}{\rightarrow}}    
\allowdisplaybreaks

\begin{document}
\title{Note on Mean Vector Testing for High-Dimensional Dependent Observations}

\author[1]{Seonghun Cho} 
\author[1]{Johan Lim}
\author[2]{Deepak Nag Ayyala\corref{augusta}\fnref{augusta}}
\ead{dayyala@augusta.edu}
\author[3]{Junyong Park}
\author[3]{Anindya Roy}

\cortext[augusta]{Corresponding author}
\address[1]{Department of Statistics, Seoul National University, Seoul, Korea}
\address[2]{Department of Population Health Sciences, Medical College of Georgia, Augusta University, Augusta, Georgia}
\address[3]{Department of Mathematics and Statistics, University of Maryland Baltimore County, Baltimore, Maryland}
\fntext[augusta]{Mailing address: Department of Population Health Sciences, Augusta University, 1120 15th St, Augusta, GA 30912.}

\begin{abstract}
\noindent  For the mean vector test in high dimension, Ayyala et al.(2017,153:136-155) proposed  new test statistics when the observational vectors are M dependent. 
Under certain conditions, the test statistics for one-same and two-sample cases 
were shown to be asymptotically normal. While the test statistics and the asymptotic results are valid,  
some parts of the proof of asymptotic normality need to be corrected.  
In this work, we provide corrections to  the  proofs of their main theorems. We also note a few minor discrepancies in calculations in the publication.
\end{abstract}

\makeatletter
\def\ps@pprintTitle{%
	\let\@oddhead\@empty
	\let\@evenhead\@empty
	\def\@oddfoot{}%
	\let\@evenfoot\@oddfoot}
\makeatother

\maketitle
\section{Introduction}
Mean vector testing in high dimension is gaining great attention in the recent past with increasing availability of data sets where the  number of variables is greater than the sample size. The traditional Hoteling's $T^2$ test and the more recently developed test statistics assume that the observations are independently and identically distributed. Comparison of mean vector for distributions where the samples are dependent is a relatively understudied problem. The one-sample mean vector test under dependence is defined as follows. The observations $X_1, \ldots ,X_n$ are assumed to be $p$-dimensional random vectors with mean $\mu$ and a covariance matrix $\Sigma$ that may not be diagonal. The problem of interest is then to test
\begin{equation}
\mathcal{H}_0: \mu = 0 \hspace{6mm} \mbox{versus} \hspace{6mm} \mathcal{H}_A: \mu \neq 0.
\label{eqn:onefammodel}
\end{equation}
In the two-sample case, $X_1^{(1)}, \ldots , X_{n_1}^{(1)}$ and $X_1^{(2)}, \ldots , X_{n_2}^{(2)}$ are two independent groups of $p$-dimensional observations with mean vectors $\mu^{(1)}$, $\mu^{(2)}$ and covariance matrices $\Sigma^{(1)}, \Sigma^{(2)}$, respectively. The hypothesis of interest is that the two population means are equal, viz.
\begin{equation}
\mathcal{H}_0: \twofam{\mu}{1} = \twofam{\mu}{2} \hspace{6mm} \mbox{versus} \hspace{6mm} \mathcal{H}_A: \twofam{\mu}{1} \neq \twofam{\mu}{2}.
\end{equation}

In \citep{ayyala2017mean}, we developed a hypothesis test for \eqref{eqn:onefammodel} when the samples follow an $M$-dependent  stationary Gaussian process with mean $\mu$ and autocovariance structure given by $\Gamma(0), \ldots, \Gamma(M)$. The rate of increase with respect to $n$ was assumed to be linear for $p$ ($p = O(n)$) and sublinear  for $M$ ($M = O(n^{1/8})$). These assumptions ensure that the number of variables is not increasing faster than the sample size and that there are sufficiently many observations for estimating the autocovariance at all lags. The dependence structure is made sparse by assuming $\tr (\Gamma(a) \Gamma(b) \Gamma(c) \Gamma(d)) = o \{ \tr^2 ( \Omega_n^2) \}$ for any $a, b, c, d \in \mathcal{M}$ where $\mathcal{M} = \{-M, \ldots, M \}$ and the matrix $\Omega_n$ is the covariance of $\xbar$ multiplied by the sample size, i.e., $\Omega_n = \msum_{h \in \mathcal{M}} n^{-1} (n - |h|) \Gamma(h)$. Summarizing the model, the assumptions are
\begin{equation}
p = O(n), \hspace{5mm} M = O(n^{1/8}), \hspace{5mm} \tr (\Gamma(a) \Gamma(b) \Gamma(c) \Gamma(d)) = o \{ \tr^2 ( \Omega_n^2) \} \,\, \forall \,\, a,b,c,d \in \mathcal{M}.
\label{eqn:assumptions}
\end{equation}

The test statistic is based on the Euclidean norm of the sample average, $\transpose{\xbar} \xbar$. Define $\mtxt{M} = \sammean - n^{-1} \widehat{\tr (\Omega_n)}$ where $\widehat{\tr (\Omega_n)}$ is an unbiased estimator for $\tr (\Omega_n)$. Under $H_0$, this quantity has expected value $\transpose{\mu} \mu$. Using this property, the test statistic is constructed as
\begin{equation}
\mtxt{T} = \frac{ \mtxt{M}}{\sqrt{\widehat{\varmn}}}, 
\label{eqn:osteststat}
\end{equation}
where $\widehat{\varmn}$ is a ratio-consistent estimator for $\varmn$, i.e.,  $\frac{\widehat{\varmn}}{\varmn} \cp 1$. For expressions of $\widehat{\tr (\Omega_n)}$ and $\widehat{\varmn}$, refer to \citet*{ayyala2017mean}. Under the null hypothesis, the test statistic is shown to be asymptotically normal. 
When we define  $a_n \simeq b_n$ if $a_n-b_n \rightarrow 0$, 
 the power function at significance level $\alpha$ under a local alternative is derived as
\begin{equation}
\beta_n(\mu) \simeq  \Phi \left( -z_{\alpha} + \frac{ n \transpose{\mu} \mu}{\sqrt{ 2 \tr (\Omega_n^2)}} \right) 
\end{equation}
as $n \rightarrow \infty$ where $\Phi$ is the cumulative distribution of standard normal distribution.

The mean in the local alternative satisfies the condition $\transpose{\mu} \{\Gamma(a) \Gamma(-a)\}^{1/2} \mu = o \{ \tr (\Omega_n^2) \}$ for all $a \in \mathcal{M}$. 

In the proof of asymptotic normality of $\mtxt{T}$, we used a two-dimensional array argument. The main idea was to divide the array into blocks which were at least $M$ indices apart. The proof was established by showing that (a) these blocks dominate the remainder of terms and (b) the blocks are independent and hence central limit theorem can be used by verifying Lyapunov conditions. However upon further inspection, we see the blocks are not independent, albeit being uncorrelated. Also, we have identified convergence issues for asymptotic power under the rate of increase of $M$ assumed in \eqref{eqn:assumptions}. The main result and numerical studies of \citet*{ayyala2017mean} remain valid and provide the justification of the use of the test statistic. However, some of the theoretical details of the paper need to be revised.  

In this article, we address these issues and provide corrected proofs for the results in \citet*{ayyala2017mean} as well as the further discussion on the power of the test. 
The remainder of the paper is organized as follows. In Section \ref{sec:normality}, a corrected proof for asymptotic normality for the one-sample test statistic is provided. Using the corrected proof, the result for the two-sample case is verified in Section \ref{sec:twosamptest}. Asymptotic power of the test statistic for the one sample case is presented in Section \ref{sec:power}. The notation used in the remainder of the article is made to be consistent with \citet*{ayyala2017mean}. For definitions and additional details on the variables defined, kindly refer to \citet*{ayyala2017mean}.

\section{Proof of theorem 3.1 of \citet*{ayyala2017mean}}
\label{sec:normality}

In the proof of theorem 3.1 of \citet*{ayyala2017mean}, the authors  claim that the variable $B_{ij}$ defined for all $i,j \in \{ 1,\ldots,k_{n} \}$ (formal 
definition of $B_{ij}$ is given in Section \ref{sec:theoremproof}) are mutually independent since the process is Gaussian with $M$-dependent stationarity. Using the independence of $B_{ij}$'s, the authors established  that the Lyapunov condition, a sufficient condition for the central limit theorem to hold for the leading term of the proposed statistic, and thus the testing statistic was shown to be asymptotically normal. Upon further inspection, we find that the variable $B_{ij}$'s are not mutually 
independent and hence a dependent central limit theorem is needed to show the asymptotic normality of the proposed statistic.

Here, we present a new proof for the asymptotic normality of the leading term $\sum_{i,j=1}^{k_{n}} B_{ij}$. Following the proof of the theorem 1 of \citet*{chen2010two}, we exploit the martingale CLT \citep[Corollary 3.1]{hall2014martingale}. Before the proof, we rewrite some assumptions and notations of \citet*{ayyala2017mean}.

\subsection{Assumptions and notations} 
\label{sec:assumptions}
The authors assume the following:
\begin{enumerate}[label=(A\arabic*)]
\item \label{assump:stationarity} $X_{1},\ldots,X_{n} \in \mathbb{R}^{p}$ follow an $M$-dependent strictly stationary Gaussian process with mean $\mu$ and autocovariance structure given by $\Gamma(0),\ldots,\Gamma(M)$. That is, for $h \in \mathcal{M} = \{ -M,\ldots,M \}$
\begin{equation} \nonumber
{\rm Cov}(X_{t},X_{t+h}) = \Gamma(h).
\end{equation}
\item \label{assump:p} $p = O(n)$ and $p \rightarrow \infty$.
\item \label{assump:M} $M = O(n^{1/8})$.
\item \label{assump:Omega} $\Omega_{n} = n{\rm Cov}(\bar{X}_{n}) = \sum_{h \in \mathcal{M}} \left( 1-\frac{|h|}{n} \right) \Gamma(h) \rightarrow F(0)$ where $F(0)$ is the spectral matrix evaluated at the zero frequency.
\item \label{assump:trace} ${\rm tr}\left( \Gamma(a)\Gamma(b)\Gamma(c)\Gamma(d) \right) =   o\left(   (M+1)^{-4}     {\rm tr}^{2}\left( \Omega_{n}^{2} \right) \right)$ where the rate of decay is uniform for all $a,b,c,d \in \mathcal{M}$.
\end{enumerate}
The naive sample estimator of the autocovariance matrix $\Gamma(h)$ at lag $h \in \mathcal{M}_{+} = \{ 0,\ldots,M \}$ is denoted by
\begin{equation} \nonumber
\widehat{\Gamma}(h) = \frac{1}{n}\sum_{t=1}^{n-h}\left(X_{t}-\bar{X}_{n}\right)\left(X_{t+h}-\bar{X}_{n}\right)^{\rm T},
\end{equation}
with $\widehat{\Gamma}(-h) = \widehat{\Gamma}(h)^{\top}$. To construct an unbiased estimator of ${\rm tr}\left(\Omega_{n}\right)$, let
\begin{equation} \nonumber
\gamma = \begin{pmatrix}
{\rm tr}(\Gamma(0)) \\ \vdots \\ {\rm tr}(\Gamma(M))
\end{pmatrix} \in \mathbb{R}^{M + 1}, \quad\mbox{and}\quad \widehat{\gamma}_{n} = \begin{pmatrix}
{\rm tr}(\widehat{\Gamma}(0)) \\ \vdots \\ {\rm tr}(\widehat{\Gamma}(M))
\end{pmatrix} \in \mathbb{R}^{M + 1}.
\end{equation}
Then, the expected value of $\widehat{\gamma}_{n}$ is ${\rm E}\left( \widehat{\gamma}_{n} \right) = \Theta_{n} \gamma$ where $\Theta_{n} \in \mathbb{R}^{\mathcal{M}_{+}\times\mathcal{M}_{+}}$ is a coefficient matrix. Since ${\rm tr}\left(\Omega_{n}\right) = b_{n}^{\rm T}\gamma$ where $b_{n} \in \mathbb{R}^{\mathcal{M}_{+}}$ with $b_{n}(0) = 1$ and $b_{n}(h) = 2\left(1 - \frac{h}{n}\right)$ for $h = 1,\ldots,M$,
\begin{equation} \nonumber
\widehat{{\rm tr}\left(\Omega_{n}\right)} = \beta_{n}^{\rm T} \widehat{\gamma}_{n}
\end{equation}
is an unbiased estimator of ${\rm tr}\left( \Omega_{n} \right)$ where $\beta_{n} = \Theta_{n}^{-1}b_{n} \in \mathbb{R}^{\mathcal{M}_{+}}$. Therefore, we have
\begin{equation} \label{eq:Mn}
E(\mtxt{M}) = E\left( \bar{X}_{n}^{\rm T}\bar{X}_{n} - \frac{1}{n} 
\widehat{{\rm tr}\left(\Omega_{n}\right)} \right) = \mu^{\rm T}\mu.
\end{equation}
According to \citet*{ayyala2017mean}, the limiting expression of the variance of $\mtxt{M}$ is expressed as 
\begin{equation} \label{eq:varMn}
{\rm var}(\mtxt{M}) = \frac{2}{n^{2}}{\rm tr}\left( \Omega_{n}^{2} \right) + o\left( \frac{1}{n^{2}}{\rm tr}\left( \Omega_{n}^{2} \right) \right).
\end{equation}

\subsection{Theorem 3.1 of  \citet*{ayyala2017mean}} 
\label{sec:theoremproof}
The main theorem of \citet*{ayyala2017mean} is stated as follows: 

\begin{theorem} \label{thm:ayyala} {\rm \citep*[theorem 3.1]{ayyala2017mean}}
Suppose that the assumptions \ref{assump:stationarity}$\sim$\ref{assump:trace} and the null hypothesis $H_{0}$ hold. Then, as $n \rightarrow \infty$,
\begin{equation} \label{eq:CLT}
T_{n} = \frac{\mtxt{M}}{\sqrt{\varmn}} \stackrel{d}{\rightarrow} N(0,1).
\end{equation}
\label{thm:theorem21}
\end{theorem}

To show the asymptotic normality of $T_{n}$, \citet*{ayyala2017mean} decompose $T_{n}$ as
\begin{equation} \label{eq:TnDecomp}
T_{n} = \frac{M_{n}}{\sqrt{{\rm var}\left(\mtxt{M}\right)}} = \frac{\bar{X}_{n}^{\rm T}\bar{X}_{n} - \frac{1}{n} {\rm tr}\left(\Omega_{n}\right)}{\sqrt{{\rm var}\left(\mtxt{M}\right)}} - \frac{\widehat{{\rm tr}\left(\Omega_{n}\right)} - {\rm tr}\left(\Omega_{n}\right)}{n\sqrt{{\rm var}\left(\mtxt{M}\right)}} = \Delta_{1} - \Delta_{2}
\end{equation}
where
\begin{equation} \label{eq:delta1}
\Delta_{1} = \frac{\bar{X}_{n}^{\rm T}\bar{X}_{n} - \frac{1}{n} {\rm tr}\left(\Omega_{n}\right)}{\sqrt{\varmn}}
\end{equation}
and
\begin{equation} \label{eq:delta2}
\Delta_{2} = \frac{\widehat{{\rm tr}\left(\Omega_{n}\right)} - {\rm tr}\left(\Omega_{n}\right)}{n\sqrt{\varmn}},
\end{equation}
and show that $\Delta_{1} \stackrel{d}{\rightarrow} N(0,1)$ and $\Delta_{2} \stackrel{p}{\rightarrow} 0$. 
To establish the asymptotic normality of $\Delta_{1}$, \citet*{ayyala2017mean} define the following quantities: 
an $n \times n$ matrix $\mathcal{A}$ with $(i,j)$th element 
\begin{equation} \nonumber
\mathcal{A}_{ij} = \frac{1}{n^{2}}\left[ X_{i}^{\rm T}X_{j} - {\rm tr}\left( \Gamma(i-j) \right) \right],
\end{equation}
where  $\sum\sum_{i,j} \mathcal{A}_{ij} =\bar{X}_{n}^{\rm T}\bar{X}_{n} - \frac{1}{n} {\rm tr}\left(\Omega_{n}\right)$;  
constants $\alpha \in (0,1)$ and $C>0$ such that $w_{n} = C \cdot n^{\alpha}$, and represent 
 $n = w_{n}k_{n}+r_{n}$ 
where $0 \le r_{n} < w_{n}$; 
for $i,j \in \{ 1,2,\ldots,k_{n} \}$, they define the random variables
\begin{eqnarray}
B_{ij} & = & \sum_{k=(i-1)w_{n}+1}^{iw_{n}-M}\sum_{l=(j-1)w_{n}+1}^{jw_{n}-M} \mathcal{A}_{kl}, \nonumber\\
D_{ij} & = & \sum_{k=(i-1)w_{n}+1}^{iw_{n}}\sum_{l=(j-1)w_{n}+1}^{jw_{n}} \mathcal{A}_{kl} - B_{ij}, \nonumber\\
F & = & \sum_{(k,l) \in \{1,\ldots,n\}^{2} - \{1,\ldots,w_{n}k_{n}\}^{2}} \mathcal{A}_{kl}. \nonumber
\end{eqnarray}
With these, \citet*{ayyala2017mean}  further decompose $\Delta_{1}$ as
\begin{eqnarray}
\Delta_{1} & = & \frac{\bar{X}_{n}^{\rm T}\bar{X}_{n} - \frac{1}{n} {\rm tr}\left(\Omega_{n}\right)}{\sqrt{\varmn}} = \frac{\sum_{i=1}^{n}\sum_{j=1}^{n} \mathcal{A}_{ij}}{\sqrt{\varmn}} \nonumber\\
& = & \frac{\sum_{i=1}^{k_{n}}\sum_{j=1}^{k_{n}} B_{ij}}{\sqrt{\varmn}} + \frac{\sum_{i=1}^{k_{n}}\sum_{j=1}^{k_{n}} D_{ij} + F}{\sqrt{\varmn}} = \Delta_{11}+\Delta_{12}, \label{eq:delta1Decomp}
\end{eqnarray}
and show that $\Delta_{11} \stackrel{d}{\rightarrow} N(0,1)$ and $\Delta_{12} \stackrel{p}{\rightarrow} 0$. While convergence of $\Delta_{12}$ can be established esasily, asymptotic normality of $\Delta_{11}$ is not straightforward.

\subsection{Revised proof} 

In \citet*{ayyala2017mean}, the authors  assumed that the blocks $B_{ij}$'s are independently distributed with unequal covariances. However, they are not independent even though  the blocks have zero covariance, ${\rm cov}(B_{ij}, B_{i' j'}) = 0$ if $(i,j) \neq (i',j')$. The dependence structure of the blocks should be taken into account in establishing limiting distribution. We provide a new proof for asymptotic normality of $\Delta_{11}$ which takes into account the dependence between the $B_{ij}$ blocks. 

\begin{proposition} \label{prop:main}
Under the assumption of Theorem \ref{thm:ayyala}, it holds that
\begin{equation} \label{eq:CLT2}
\Delta_{11} = \frac{\sum_{i=1}^{k_{n}}\sum_{j=1}^{k_{n}} B_{ij}}{\sqrt{{\rm Var}\left(M_{n}\right)}} \stackrel{d}{\rightarrow} N(0,1)
\end{equation}
as $n \rightarrow \infty$.
\end{proposition}

\begin{proof}
For $i=1,\ldots,k_{n}$, let
\begin{equation} \nonumber
Y_{i} = \frac{1}{w_{n}-M}\sum_{k=(i-1)w_{n}+1}^{iw_{n}-M}X_{k}.
\end{equation}
Then, $Y_{1},\ldots,Y_{k_{n}}$ are independent and identically distributed Gaussian variables with mean vector 0 and covariance matrix ${\rm Cov}(Y_{i}) = \frac{1}{w_{n}-M} \Omega_{w_{n}}$ where $\Omega_{w_{n}} = \sum_{h \in \mathcal{M}}\left(1-\frac{|h|}{w_{n}-M}\right)\Gamma(h)$. We note that for $i \neq j$,
\begin{equation} \nonumber
B_{ij} = \sum_{k=(i-1)w_{n}+1}^{iw_{n}-M}\sum_{l=(j-1)w_{n}+1}^{jw_{n}-M}X_{k}^{\rm T}X_{l} =
 {{(w_{n}-M)^{2}}Y_{i}^{\rm T}Y_{j}.}
\end{equation}
Following the proof of Theorem 1 of \citet*{chen2010two}, define
\begin{eqnarray}
\phi_{nij} & = & \frac{(w_{n}-M)^{2}}{n^{2}} Y_{i}^{\rm T}Y_{j} \quad \mbox{for } i<j, \nonumber\\
V_{nj} & = & \sum_{i=1}^{j-1} \phi_{nij} \quad \mbox{for } j = 2,\ldots,k_{n}, \nonumber\\
S_{nm} & = & \sum_{j=2}^{m} V_{nj} \quad \mbox{for } m = 2,\ldots,k_{n}
\end{eqnarray}
and let $\mathcal{F}_{nm}  =  \sigma(Y_{1},\ldots,Y_{m})$ for  $m = 2,\ldots,k_{n}$ be the $\sigma$-algebra generated by $Y_{1},\ldots,Y_{m}$. Then we have  
\begin{equation} \nonumber
\sum_{i \neq j} B_{ij} = 2S_{nk_{n}}.
\end{equation}
To show the asymptotic normality of $\sum_{i \neq j} B_{ij}$, we need some lemmas.

\begin{lemma}[square integrable martingale with zero mean] \label{lem:martigale}
For each $n$, $(S_{nm},\mathcal{F}_{nm})_{m=2}^{k_{n}}$ is the sequence of zero mean and a square integrable martingale.
\end{lemma}

\begin{proof}[Proof of lemma \ref{lem:martigale}]
Since $Y_{i}$'s are i.i.d. Gaussian random variables with zero mean, $S_{nm}$ has zero mean and is square integrable for any $n, m$. We observe
\begin{eqnarray}
\mathrm{E}\left(S_{n(m+1)} | \mathcal{F}_{nm}\right) & = & \sum_{j=2}^{m+1} \sum_{i=1}^{j-1} \frac{(w_{n}-M)^{2}}{n^{2}} \mathrm{E}\left(Y_{i}^{\rm T}Y_{j} | \mathcal{F}_{nm}\right) \nonumber\\
& = & \sum_{j=2}^{m} \sum_{i=1}^{j-1} \frac{(w_{n}-M)^{2}}{n^{2}} Y_{i}^{\rm T}Y_{j} + \sum_{i=1}^{m} \frac{(w_{n}-M)^{2}}{n^{2}} Y_{i}^{\rm T}\mathrm{E}\left(Y_{m+1} | \mathcal{F}_{nm}\right) \nonumber\\
& = & S_{nm}. \nonumber
\end{eqnarray}
Therefore, $(S_{nm},\mathcal{F}_{nm})_{m=2}^{k_{n}}$ is a 
zero mean and square integrable martingale sequence.
\end{proof}

\begin{lemma}[analogous condition on the conditional variance] \label{lem:condvar}
Let $\sigma_{n}^{2} = \frac{2k_{n}(k_{n}-1)(w_{n}-M)^{2}}{n^{4}} {\rm tr}\left(\Omega_{w_{n}}^{2}\right)$. Then,
\begin{equation} \nonumber
\frac{1}{\sigma_{n}^{2}} \sum_{j=2}^{k_{n}}{\rm E}\left[V_{nj}^{2}|\mathcal{F}_{n(j-1)}\right] \stackrel{p}{\rightarrow} \frac{1}{4}.
\end{equation}
\end{lemma}

\begin{proof}[Proof of lemma \ref{lem:condvar}]
We note that
\begin{eqnarray}
{\rm E}\left[V_{nj}^{2}|\mathcal{F}_{n(j-1)}\right] & = & \frac{(w_{n}-M)^{4}}{n^{4}} {\rm E}\left[\left.\left( \sum_{i=1}^{j-1}Y_{i}^{\rm T}Y_{j} \right)^{2}\right|\mathcal{F}_{n(j-1)}\right] \nonumber\\
& = & \frac{(w_{n}-M)^{4}}{n^{4}} \sum_{i_{1},i_{2}=1}^{j-1} Y_{i_{1}}^{\rm T} {\rm E}\left[Y_{j}Y_{j}^{\rm T}|\mathcal{F}_{n(j-1)}\right] Y_{i_{2}} \nonumber\\
& = & \frac{(w_{n}-M)^{3}}{n^{4}} \sum_{i_{1},i_{2}=1}^{j-1} Y_{i_{1}}^{\rm T} \Omega_{w_{n}} Y_{i_{2}}. \nonumber
\end{eqnarray}
Define $\eta_{n} = \sum_{j=2}^{k_{n}}{\rm E}\left[V_{nj}^{2}|\mathcal{F}_{n(j-1)}\right]$, then its first moment is
\begin{eqnarray}
{\rm E}\left( \eta_{n} \right) & = & \frac{(w_{n}-M)^{3}}{n^{4}} \sum_{j=2}^{k_{n}} \sum_{i_{1},i_{2}=1}^{j-1} {\rm E}\left( Y_{i_{1}}^{\rm T} \Omega_{w_{n}} Y_{i_{2}} \right) \nonumber\\
& = & \frac{(w_{n}-M)^{2}}{n^{4}} \sum_{j=2}^{k_{n}} (j-1){\rm tr}\left(\Omega_{w_{n}}^{2}\right) \nonumber\\
& = & \frac{k_{n}(k_{n}-1)(w_{n}-M)^{2}}{2n^{4}} {\rm tr}\left(\Omega_{w_{n}}^{2}\right) = \frac{1}{4}\sigma_{n}^{2}. \nonumber
\end{eqnarray}
The second moment of $\eta_{n}$ is
\begin{eqnarray}
{\rm E}\left( \eta_{n}^{2} \right) & = & \frac{(w_{n}-M)^{6}}{n^{8}} \sum_{j_{1},j_{2}=2}^{k_{n}} \sum_{i_{1},i_{2}=1}^{j_{1}-1}\sum_{i_{3},i_{4}=1}^{j_{2}-1} {\rm E}\left( Y_{i_{1}}^{\rm T} \Omega_{w_{n}} Y_{i_{2}}Y_{i_{3}}^{\rm T} \Omega_{w_{n}} Y_{i_{4}} \right) \nonumber\\
& = & \frac{(w_{n}-M)^{6}}{n^{8}} \sum_{j_{1},j_{2}=2}^{k_{n}} (j_{1}-1)^{2}(j_{2}-1)^{2} {\rm E}\left( \bar{Y}_{j_{1}}^{\rm T} \Omega_{w_{n}} \bar{Y}_{j_{1}}\bar{Y}_{j_{2}}^{\rm T} \Omega_{w_{n}} \bar{Y}_{j_{2}} \right) \nonumber\\
& = & \frac{(w_{n}-M)^{4}}{n^{8}} \sum_{j_{1},j_{2}=2}^{k_{n}} (j_{1}-1)^{2}(j_{2}-1)^{2} \left[ \frac{2(j_{1} \wedge j_{2} - 1)^{2}}{(j_{1}-1)^{2}(j_{2}-1)^{2}}{\rm tr}\left(\Omega_{w_{n}}^{4}\right) + \frac{1}{(j_{1}-1)(j_{2}-1)}{\rm tr}^{2}\left(\Omega_{w_{n}}^{2}\right) \right] \nonumber\\
& = & \frac{(w_{n}-M)^{4}}{n^{8}} \left[ \frac{k_{n}^{2}(k_{n}-1)(k_{n}+1)}{6}{\rm tr}\left(\Omega_{w_{n}}^{4}\right) + \frac{k_{n}^{2}(k_{n}-1)^{2}}{4}{\rm tr}^{2}\left(\Omega_{w_{n}}^{2}\right) \right], \nonumber
\end{eqnarray}
and thus the variance of $\eta_{n}$ is
\begin{equation} \nonumber
{\rm Var}\left( \eta_{n} \right) = \frac{k_{n}^{2}(k_{n}-1)(k_{n}+1)(w_{n}-M)^{4}}{6n^{8}} {\rm tr}\left(\Omega_{w_{n}}^{4}\right).
\end{equation}
{By assumption (A5), it holds that ${\rm tr}\left(\Omega_{w_{n}}^{4}\right) = o\left( {\rm tr}^{2}\left(\Omega_{w_{n}}^{2}\right) \right)$.} Thus, we have
\begin{equation} \nonumber
\frac{{\rm Var}\left( \eta_{n} \right)}{\sigma_{n}^{4}} = \frac{k_{n}+1}{24(k_{n}-1)} \frac{{\rm tr}\left(\Omega_{w_{n}}^{4}\right)}{{\rm tr}^{2}\left(\Omega_{w_{n}}^{2}\right)} = o(1).
\end{equation}
This completes the proof.
\end{proof}

\begin{lemma}[conditional Lindeberg condition] \label{lem:condLindeberg}
For any $\epsilon > 0$,
\begin{equation} \nonumber
\frac{1}{\sigma_{n}^{2}} \sum_{j=2}^{k_{n}}{\rm E}\left[V_{nj}^{2}I\left(|V_{nj}|>\epsilon\sigma_{n}\right)|\mathcal{F}_{n(j-1)}\right] \stackrel{p}{\rightarrow} 0.
\end{equation}
\end{lemma}

\begin{proof}[Proof of lemma \ref{lem:condLindeberg}]
We note that
\begin{equation} \nonumber
\frac{1}{\sigma_{n}^{2}} \sum_{j=2}^{k_{n}}{\rm E}\left[V_{nj}^{2}I\left(|V_{nj}|>\epsilon\sigma_{n}\right)|\mathcal{F}_{n(j-1)}\right] \le \frac{1}{\epsilon^{2}\sigma_{n}^{4}} \sum_{j=2}^{k_{n}}{\rm E}\left[V_{nj}^{4}|\mathcal{F}_{n(j-1)}\right].
\end{equation}
It suffices to show
\begin{equation} \nonumber
{\rm E} \left[ \sum_{j=2}^{k_{n}}{\rm E}\left[V_{nj}^{4}|\mathcal{F}_{n(j-1)}\right] \right] = o(\sigma_{n}^{4}).
\end{equation}
We have that
\begin{eqnarray}
{\rm E} \left[ \sum_{j=2}^{k_{n}}{\rm E}\left[V_{nj}^{4}|\mathcal{F}_{n(j-1)}\right] \right] & = & \sum_{j=2}^{k_{n}} {\rm E}\left(V_{nj}^{4}\right) =  \sum_{j=2}^{k_{n}} {\rm E}\left( \frac{(w_{n}-M)^{2}}{n^{2}}\sum_{i=1}^{j-1}Y_{i}^{\rm T}Y_{j} \right)^{4}\nonumber\\
& = & \frac{(w_{n}-M)^{8}}{n^{8}} \sum_{j=2}^{k_{n}} \left[ \frac{6(j-1)^{2}}{(w_{n}-M)^{4}}{\rm tr}\left( \Omega_{w_{n}}^{4} \right) + \frac{3(j-1)^{2}}{(w_{n}-M)^{4}}{\rm tr}^{2}\left( \Omega_{w_{n}}^{2} \right) \right] \nonumber\\
& = & \frac{k_{n}(k_{n}-1)(2k_{n}-1)(w_{n}-M)^{4}}{2n^{8}} \left( 2{\rm tr}\left( \Omega_{w_{n}}^{4} \right) + {\rm tr}^{2}\left( \Omega_{w_{n}}^{2} \right) \right) \nonumber\\
& = & o(\sigma_{n}^{4}). \nonumber
\end{eqnarray}
This completes the proof.
\end{proof}
We  have thus established the sufficient conditions for martingale Central Limit Theorem. By the Corollary 3.1 of \citet*{hall2014martingale}, Lemmas \ref{lem:martigale}, \ref{lem:condvar} and \ref{lem:condLindeberg}, 
\begin{equation} \nonumber
\Delta_{111} = \frac{\sigma_{n}}{\sqrt{{\rm Var}\left(M_{n}\right)}}\frac{\sum_{i \neq j}B_{ij}}{\sigma_{n}} \stackrel{d}{\rightarrow} N(0,1).
\end{equation}
We define  
\begin{eqnarray}
a_n \asymp b_n
\label{eqn:asymp}
\end{eqnarray}
 if $ a_n =b_n(1+o(1))$ for $a_n>0$ and $b_n>0$.
Note that  $\frac{\sigma_n}{\sqrt{Var(M_n)}} \rightarrow 1$ since 
	$\frac{\sigma_n^2}{Var(M_n)}  -1 \asymp \frac{ tr(\Omega_{w_n}^2)   - tr(\Omega_{n}^2)}{tr(\Omega_n^2)}$ and 
\begin{eqnarray}
  \frac{ |tr(\Omega_{w_n}^2)   - tr(\Omega_{n}^2)|}{tr(\Omega_n^2)} &=& \frac{|\sum_{|i|,|j| \leq M} \left( -\frac{2|i|}{w_n-M} + \frac{2|j|}{n}  + \frac{|i||j|}{(w_n-M)^2}  - \frac{|i||j|}{n^2}  \right) tr(\Gamma(i) \Gamma(j))| }{tr(\Omega_n^2)}     \nonumber  \\
 & \leq  &   4 \frac{4(M+1)^3}{w_n(1+o(1))} \sqrt{n} \frac{1}{(M+1)^2}  \frac{o(tr(\Omega_n^2))}{tr(\Omega_n^2)}    \label{eqn:inequality} \\
  &=& o(1)  \label{eqn:equality}  
\end{eqnarray}  
where $\leq$ in \eqref{eqn:inequality} is due to  $\Gamma(i)\Gamma(j) = o(\sqrt{n} tr(\Omega_n^2))$ from Lemma C.1 in the appendix of \citet{ayyala2017mean} and the equality in \eqref{eqn:equality} is obtained by taking   $w_n  \geq  (M+1) \sqrt{n}$. 
Also,  by the classical central limit theorem, it holds that $\frac{\sum_{i=1}^{k_{n}} B_{ii}}{\sqrt{k_{n}{\rm Var}(B_{11})}} \stackrel{d}{\rightarrow} N(0,1)$. Since
\begin{equation} \nonumber
{\rm Var}(B_{11}) = \frac{2(w_{n}-M)^{2}}{n^{4}} {\rm tr}\left( \Omega_{w_{n}}^{2} \right),
\end{equation}
we have
\begin{equation} \nonumber
\Delta_{112} = \frac{1}{\sqrt{k_{n}}} \sqrt{ \frac{k_{n}^{2}{\rm Var}(B_{11})}{{\rm Var}\left(M_{n}\right)} } \frac{\sum_{i=1}^{k_{n}} B_{ii}}{\sqrt{k_{n}{\rm Var}(B_{11})}} \stackrel{p}{\rightarrow} 0.
\end{equation}
Therefore, we get the desired result 
\begin{equation} \nonumber
\Delta_{11} = \Delta_{111}+\Delta_{112} \stackrel{d}{\rightarrow} N(0,1).
\end{equation}
\end{proof}

\section{Two Sample Test}
\label{sec:twosamptest} 
 By construction of the blocks, we have $ \frac{w_{n_g} -M }{n_g} \asymp  \frac{1}{k_{n_g}}$ from ${n_g} = k_{n_g} w_{n_g} + r_{n_g}$ with 
$w_{n_g}$ satisfying $M=o(w_{n_g})$ for $g=1,2$ 
where $\asymp$ is defined in \eqref{eqn:asymp}.   
 As done in one sample problem, if we define 
\begin{equation} \nonumber
	B_{gij} = \sum_{k=(i-1)w_{n_g}+1}^{iw_{n_g}-M}\sum_{l=(j-1)w_{n_g}+1}^{jw_{n_g}-M}X_{gk}^{\top}X_{gl} = (w_{n_g}-M)^{2}Y_{gi}^{\top}Y_{gj}, 
\end{equation}
where $Y_{gi} = \frac{1}{w_{n_g} - M}\sum_{k=(i-1)w_{n_g}+1}^{iw_{n_g}-M} X_{gk}$,  
then  two sample test statistic is 
\begin{align*}
	M_n = & (\bar X_1 - \bar X_2)'(\bar X_1 -\bar X_2)  - \frac{1}{n_1} tr( \widehat{\Omega_{n_1}^{(1)}}) -\frac{1}{n_2} tr( \widehat{\Omega_{n_2}^{(2)}})\\
	= & \frac{(w_{n_1}-M)^{2}}{n_1^{2}} \sum_{i \neq j} Y_{1i}^{\top}Y_{1j} +  \frac{(w_{n_2}-M)^{2}}{n_2^{2}}  \sum_{i\neq j} Y_{2i}^{\top}Y_{2j} -  2  \frac{(w_{n_1}-M)(w_{n_2}-M) }{n_1n_2} \sum_{1\leq i \leq k_{n_1}}\sum_{1\leq j \leq k_{n_2}}Y_{1i}^{\top}Y_{2j}     \\
	& + {\cal R}_n\\
	= & \underbrace{ \left\{ \frac{1}{k_{n_1}(k_{n_1}-1)} \sum_{i \neq j} Y_{1i}^{\top}Y_{1j} +  \frac{1}{k_{n_2}(k_{n_2}-1)}  \sum_{i\neq j} Y_{2i}^{\top}Y_{2j} -  2  \frac{1}{k_{n_1}k_{n_2}} \sum_{1\leq i \leq k_{n_1}}\sum_{1\leq j \leq k_{n_2}}Y_{1i}^{\top}Y_{2j} \right\} }_{ T_n }   (1+o_p(1))    \\
	&+ {\cal R}_n  \\
	=& T_n(1+o_o(1)) +{\cal R}_n 
\end{align*}
where  
\begin{eqnarray}
	{\cal R}_n =  \sum_{i=1}^{k_{n_1}} B_{1ii} 
	+  \sum_{i=1}^{k_{n_2}} B_{2ii} + \Delta_{12}^{(1)} + \Delta_{12}^{(2)}
	\label{eqn:Rn}
\end{eqnarray} 
and $\Delta_{12}^{(l)}$ for $l=1,2$ is defined as in one sample case. 

We change the condition (14) in Ayyala et al. (2017) to the following:  
\begin{eqnarray}
tr(\Gamma^{(w_1)}(a)\Gamma^{(w_2)}(b)\Gamma^{(w_3)}(c)\Gamma^{w_4}(d))
= o((M+1)^{-4}  tr^2(  (\Omega^{(1)}_{n_1} + \Omega^{(2)}_{n_2})^2) )
\end{eqnarray}

As in one sample case,  we can show  
\begin{eqnarray}
	\frac{{\cal R}_n }{\sqrt{ \varmn }} \stackrel{p}{\rightarrow} 0.
\end{eqnarray}

We now prove the asymptotic normality of $T_n$. 
Since $Y_{11}, Y_{12},\ldots, Y_{1 k_{n_1}} $ and $Y_{21}, Y_{22},\ldots, Y_{2 k_{n_2}}$ are $k_{n_1} + k_{n_2}$ observational vectors  
with $E(Y_{gi}) =\mu_{g}$ and 
$Cov(Y_{gi}) = \frac{1}{(w_{n_g}-M)} \Omega^{(g)}_{w_{n_g}}$ for $g=1,2$, 
   we can  apply the martingale C.L.T. in Chen and Qin (2010) to independent $k_{n_1} + k_{n_2}$ observations. Therefore we obtain     
\begin{eqnarray}
	\frac{T_n}{\sqrt{{\rm var}(T_n)}} \stackrel{d}{\rightarrow} N(0,1)
	\label{eqn:Tn}
\end{eqnarray}
where 
\begin{align}
	{\rm var}(T_n) &=   \frac{2}{k_{n_1} (k_{n_1}-1)}  \frac{1}{(w_{n_1}  -M)^2} tr((\Omega^{(1)}_{w_{n_1}})^2)  
	+	\frac{2}{k_{n_2} (k_{n_2}-1)}  \frac{1}{(w_{n_2}  -M)^2} tr((\Omega^{(2)}_{w_{n_2}})^2) \nonumber \\
	 & +  \frac{4}{k_{n_1} k_{n_2}}  \frac{1}{(w_{n_1}  -M)(w_{n_1}  -M )} tr(\Omega^{(1)}_{w_{n_1}}\Omega^{(2)}_{w_{n_2}}   ) \nonumber \\
	&\asymp    \frac{2}{{n_1} ({n_1}-1)}   tr((\Omega^{(1)}_{w_{n_1}})^2)   	+
	\frac{2}{{n_2} ({n_2}-1)}  tr((\Omega^{(2)}_{w_{n_1}})^2) +  \frac{4}{{n_1}{n_2}}   tr(\Omega^{(1)}_{w_{n_1}}\Omega^{(2)}_{w_{n_2}} )             
\end{align}
since $ M=o(w_{n_g})$ and $ k_{n_g} w_{n_g} \asymp n_g$.    
Furthermore,  we can also show the asymptotic equivalence of ${\rm var}(T_n)$ and $\varmn$ which is stated in the following lemma. 

\begin{lemma}
Under the conditions given in Ayyala et al. (2017),  we have $ {\rm var}(T_n) \asymp \varmn $.  
\label{lemma:ratioVarMnTn}	
\end{lemma} 
\begin{proof}
		 By using  the following facts: 
\begin{itemize}
	\item  when $n_1$ and $n_2$ have the same order except constant factor,     $ \varmn \asymp  tr( \frac{\Omega_{n_1}}{n_1} + \frac{\Omega_{n_2}}{n_2})^2$ has the same order as 
	$\frac{1}{n^2} tr(({\Omega_{n_1}} +{\Omega_{n_2}})^2) $ except some constant factor. This can be show using  $tr(\Omega_{n_l}^2) >0$ for $l=1,2$ and $tr(\Omega_{n_1} \Omega_{n_2})\geq 0$         
	
	\item when $p=O(n)$,   $\Gamma(i)\Gamma(j) = \sqrt{p} o( tr( \Omega_{n_1} + \Omega_{n_2} )^2  ) $ in Lemma C.1 in Appendix in Ayyala et al. (2017),   
\end{itemize}
 we can derive   
\begin{eqnarray*}
	\frac{{\rm var}(T_n)}{\varmn} - 1 &=& 
	\frac{ \sum_{|i|, |j| \leq M} (\frac{|i|}{w_n-M} +  \frac{|j|}{n_1}  - \frac{|i||j|}{(w_n -M)n_1}) tr(\Gamma^{(1)}(i)\Gamma^{(1)}(j)) }{Var(T_n)} \\
	&&+  \frac{ \sum_{|i|, |j| \leq M} (\frac{|i|}{w_n-M} +  \frac{|j|}{n_2}  - \frac{|i||j|}{(w_n -M)n_2}) tr(\Gamma^{(2)}(i)\Gamma^{(2)}(j)) }{Var(T_n)} \\
	&&+  \frac{ \sum_{|i|, |j| \leq M} (\frac{|i|}{w_n-M} +  \frac{|j|}{n_1}  - \frac{|i||j|}{(w_n -M)n_1}) tr(\Gamma^{(1)}(i)\Gamma^{(2)}(j)) }{Var(T_n)} \\
	&=&  \frac{M^3}{w_n} \sqrt{p} \frac{ o( tr( \Omega^{(1)}_{n_1} + \Omega^{(2)}_{n_2} )^2  )}{   tr( \Omega^{(1)}_{n_1} + \Omega^{(2)}_{n_2} )^2}  = \sqrt{p} \frac{M^3}{w_n} o( 1  )  \\
	&=&  o(1)   
\end{eqnarray*}  
where the last equality holds if $ w_n = [n^{\alpha}] $ for $ \frac{7}{8} \leq \alpha <1$.
\end{proof}

Combining \eqref{eqn:Rn}, \eqref{eqn:Tn} and Lemma \ref{lemma:ratioVarMnTn}, we have 
\begin{eqnarray}
	\frac{M_n}{\sqrt{\varmn}} =  \sqrt{\frac{{\rm var}(T_n)}{\varmn}} \frac{T_n}{\sqrt{{\rm var}(T_n)} }
	+ \frac{{\cal R}_n}{\sqrt{\varmn}} \stackrel{d}{\rightarrow} N(0,1). 
	\label{eqn:normality}
\end{eqnarray}
Additionally, since we have $\frac{\widehat{\varmn}}{\varmn} \stackrel{p}{\rightarrow} 1$ from \citet{ayyala2017mean}, we finally have the asymptotic normality of the proposed test for two sample case in \citet{ayyala2017mean}:   
\begin{eqnarray}
	\frac{M_n}{\sqrt{\widehat{\varmn}}} \stackrel{p}{\rightarrow} N(0,1). 
\end{eqnarray}

\section{Asymptotic power}
\label{sec:power}

To calculate the asymptotic power of $\mathcal{T}_n$ under the local alternative condition in equation (11), define $X _t = \mu + Y_t, t = 1, \ldots, n$ where $Y_t$ has mean zero and autocovariance structure given by $\{\Gamma(h)\}$. Define the test statistic calculated using $\mathbf{X}=\{X_1, \ldots, X_n\}$ and $\mathbf{Y}=\{Y_1, \ldots, Y_n\}$ as $\mathcal{T}_n(\mathbf{X})$ and $\mathcal{T}_n(\mathbf{Y})$ respectively. Similarly, we also define  $\mtxt{M}(\mathbf{X})$ 
and $\mtxt{M}(\mathbf{Y})$ corresponding to $\mtxt{M}$  defined in Section 3.    By Theorem 3.1, we have asymptotic normality of  $\mathcal{T}_n(\mathbf{Y})$,
\begin{equation*}
T_{new}(\mathbf{Y}) = \frac{ \mathcal{M}_n (\mathbf{Y})}{ \sqrt{ \widehat{ {\rm var}\left(\mathcal{M}_n (\mathbf{Y}) \right)}} } \stackrel{d}{\rightarrow} \mathcal{N}(0,1).
\end{equation*}
For any non-zero $\mu$ satisfying the local alternative condition, we have
\begin{align*}
{\rm T}_{new}(\mathbf{X}) & = \frac{ \mathcal{M}_n (\mathbf{X})}
{ \sqrt{ \widehat{ {\rm var}(\mathcal{M}_n (\mathbf{X}) )}}}  \\
& = \frac{1}{{ \sqrt{ \widehat{{\rm var}(\mathcal{M}_n(\mathbf{X}))}}}} \left( \overline{X}_n^{\top} \overline{X}_n - \frac{1}{n} \widehat{ {\rm tr} \{ \Omega_n(\mathbf{X}) \}} \right) \\
&  = \sqrt{ \frac{{\rm var}(\mtxt{M}({\bf Y}) )}{ \widehat{ {\rm var}(\mtxt{M} (\mathbf{X}) )}} } \frac{1}{\sqrt{  {\rm var}(\mtxt{M}({\bf Y}) )}}  \left( \overline{X}_n^{\top} \overline{X}_n - \frac{1}{n} \widehat{ {\rm tr} \{ \Omega_n(\mathbf{X}) \}} \right)
\end{align*}

By Theorem 3.2, the variance estimator is ratio-consistent even under the local alternative, hence $\widehat{ {\rm var}(\mtxt{M} (\mathbf{X}))}/{\rm var}(\mtxt{M}({\bf Y}))\stackrel{p}{\rightarrow} 1$. The matrix $\Omega_n$ is the second central moment of the data and is hence independent of $\mu$. Thus, we have $\Omega_n(\mathbf{X}) = \Omega_n(\mathbf{Y}) = \Omega_n$. Expressing $\mathcal{M}_n(\mathbf{X})$ in terms of $\mathbf{Y}$, we have
\begin{align*}
\mathcal{M}_n(\mathbf{X}) & = \overline{X}_n^{\top} \overline{X}_n - \frac{1}{n} \widehat{ {\rm tr} \left( \Omega_n(\mathbf{X}) \right)} \\
&= \left( \overline{Y}_n + \mu \right)^{\top} \left( \overline{Y}_n + \mu \right) - \frac{1}{n} \widehat{ {\rm tr} \left( \Omega_n (\mathbf{X}) \right)} \\
& = \overline{Y}_n^{\top} \overline{Y}_n - 2 \mu^{\top} \overline{Y}_n + \mu^{\top}{\mu} - \frac{1}{n} \widehat{ {\rm tr} \left( \Omega_n (\mathbf{Y}) \right)} \\
& = \mathcal{M}_n (\mathbf{Y}) - 2 \mu^{\top} \overline{Y}_n + \mu^{\top}{\mu}.
\end{align*}

The second term, $ \mu^{\top} \overline{Y}_n$ is normally distributed with $E(\mu^{\top} Y) = \mu^{\top} 0 = 0$ and variance ${\rm var} ({\mu^{\top} \overline{Y}_n}) = \frac{1}{n} \mu^{\top} \Omega_n \mu$. This gives
\begin{align}
{\rm var} \left(\frac{ \mu^{\top} \overline{Y}_n}{ \sqrt{ {\rm var}(\mathcal{M}_n({\bf Y}) )}} \right) &= \frac{ {\rm var}(\mu^{\top} \overline{Y}_n)}{{\rm var}(\mathcal{M}_n({\bf Y}))} \\
& = \frac{ n^{-1} \mu^{\top} \Omega_n \mu}{2 n^{-2} tr(\Omega_n^2)} \{1 + o(1) \},
\label{eqn:varratio}
\end{align}
where the last equality follows from Proposition 3.1. Writing $\Omega_n$ in terms of the autocovariance matrices $\Gamma(h), h \in \mathcal{M}$, we have the ratio as a linear combination of $\mu^{\top} \Gamma(h) \mu, h = 0, \pm 1, \ldots, \pm M$.

 For each $h$, we consider 
the singular value decomposition (SVD) of $\Gamma(h)$, $\Gamma(h) = U \Lambda V^{\top}$,  where $\Lambda$ is the diagonal matrix with positive square roots of eigenvalues of $ \left[ \Gamma(h) \Gamma(h)^{\top} \right]$. It yields $\Gamma(h) \Gamma(-h) = U \Lambda^2 U^{\top}$ which gives 
$U \Lambda U^{\top} = \left(\Gamma(h) \Gamma(-h) \right)^{1/2}$. Thus, 
\begin{align}
\mu^{\top} \Gamma(h) \mu =&  \mu^{\top} U \Lambda V^{\top} \mu = \left( \Lambda^{1/2} U^{\top} \mu \right)^{\top} \left( \Lambda^{1/2} V^{\top} \mu \right) \nonumber \\
& \leq  \sqrt{\mu^{\top} \left[ \Gamma(h) \Gamma(-h) \right]^{1/2} \mu} \sqrt{\mu^{\top} \left[ \Gamma(-h) \Gamma(h) \right]^{1/2} \mu}, \label{eqn:ratiovar}
\end{align}
where the inequality comes from the Cauchy-Schwarz inequality and 
\begin{align*}
\left( \Lambda^{1/2} U^{\top} \mu \right)^{\top} \left( \Lambda^{1/2} U^{\top} \mu \right) = \mu^{\top} U \Lambda U^{\top} \mu = \mu^{\top} \left(\Gamma(h) \Gamma(-h) \right)^{1/2} \mu, \\
\left( \Lambda^{1/2} V^{\top} \mu \right)^{\top} \left( \Lambda^{1/2} V^{\top} \mu \right) = \mu^{\top} V \Lambda V^{\top} \mu = \mu^{\top} \left(\Gamma(-h) \Gamma(h) \right)^{1/2} \mu.
\end{align*}
The upper bound on $\mu^{\top} \Omega_n \mu$ is given by
\begin{align}
\mu^{\top} \Omega_n \mu & = \mathop{\sum}_{h = -M}^M \left(1 - \frac{|h|}{n} \right) \mu^{\top} \Gamma(h) \mu \nonumber \\
& \leq  \mathop{\sum}_{h = -M}^M  \left(1 - \frac{|h|}{n} \right) \sqrt{\mu^{\top} \left[ \Gamma(h) \Gamma(-h) \right]^{1/2} \mu} \sqrt{\mu^{\top} \left[ \Gamma(-h) \Gamma(h) \right]^{1/2} \mu}.
\end{align}
In sequel, the ratio in \eqref{eqn:varratio} is bounded by
\begin{align}
\frac{ n^{-1}}{2n^{-2} {\rm tr}(\Omega_n^2)} \mu^{\top} \Omega_n \mu & \leq \frac{ n^{-1}}{2n^{-2} {\rm tr}(\Omega_n^2)} \mathop{\sum}_{h = -M}^M  \left(1 - \frac{|h|}{n} \right) \sqrt{\mu^{\top} \left[ \Gamma(h) \Gamma(-h) \right]^{1/2} \mu} \sqrt{\mu^{\top} \left[ \Gamma(-h) \Gamma(h) \right]^{1/2} \mu} \nonumber \\
& = \mathop{\sum}_{h = -M}^M  \left(1 - \frac{|h|}{n} \right)   \sqrt{ \frac{\mu^{\top} \left[ \Gamma(h) \Gamma(-h) \right]^{1/2} \mu}{n^{-1} {\rm tr}(\Omega_n^2)}}   \sqrt{\frac{\mu^{\top} \left[ \Gamma(-h) \Gamma(h) \right]^{1/2} \mu}{n^{-1} {\rm tr}(\Omega_n^2)}}. \label{eqn:ratiovar-2nd}
\end{align}

\citet{ayyala2017mean} make the assumption (see Eqn. (11) of \citet{ayyala2017mean})
\begin{equation} \label{eqn:assume-ayyala2017} 
\mu^{\top} \left[ \Gamma(h) \Gamma(-h) \right]^{1/2} \mu =    o( n^{-1} tr(\Omega_n^2)),   
\end{equation}
and this makes the summands in (\ref{eqn:ratiovar-2nd}) are all in the order of $o(1)$ and the summation (\ref{eqn:ratiovar-2nd})
converges to zero if $M$ is finite. In this article, we assume $M = O(n^{1/8})$ and change the condition (\ref{eqn:assume-ayyala2017}) accordingly to 
\begin{eqnarray}
\mu^{\top} \left[ \Gamma(h) \Gamma(-h) \right]^{1/2} \mu &=&    o((M+1)^{-1} n^{-1} tr(\Omega_n^2)).
\label{eqn:localalt2}
\end{eqnarray}
In (\ref{eqn:localalt2}), if $M$ is bounded, the 
condition \eqref{eqn:localalt2} equals to the condition (\ref{eqn:assume-ayyala2017}) that is in \citet{ayyala2017mean}.
When $M$ has the order of $n^{1/8}$, the order in (\ref{eqn:localalt2}) becomes  $o \left ( n^{-9/8} \tr \left( \Omega_n^2 \right) \right)$.   
With the assumptions $M = O(n^{1/8})$  and (\ref{eqn:localalt2}), 
we have the variance of $\mu^{\top} \overline{Y}_n/\sqrt{ {\rm var}{\mathcal{M}_n ({\bf Y})}}$ converging to zero. Further,
 this quantity has expected value zero, we claim that it converges to zero in probability.

\color{black}
Combining asymptotic normality of $\mathcal{M}_n(\mathbf{Y})$ and the result above, we can show that 
\begin{align*}
1 - \beta = P \left( \mathcal{M}_n (\mathbf{Y}) > -z_{\alpha} \right)  & = P \left( \frac{\mathcal{M}_n (\mathbf{Y}) - 2 \mu^{\top} \overline{Y}_n + \mu^{\top}{\mu}}{\sqrt{2 n^{-2} \tr (\Omega_n^2)}}  > z_{\alpha} \right) \\
& \simeq P \left( \frac{\mathcal{M}_n (\mathbf{Y}) +  \mu^{\top}{\mu}}{\sqrt{2 n^{-2} \tr (\Omega_n^2)}}  > z_{\alpha} \right) \\
& \simeq P \left(  Z > z_{\alpha} - \frac{\mu^{\top}{\mu}}{\sqrt{2 n^{-2} \tr (\Omega_n^2)}} \right) \\
& \simeq \Phi \left( -z_{\alpha} + \frac{n \mu^{\top}{\mu}}{\sqrt{2 \tr (\Omega_n^2)}} \right).
\end{align*}

\section{Concluding remarks}  

We conclude the note with a few additional remarks. First, for practical use, the denominator of the test statistic in 
 of Theorem \ref{thm:theorem21} 
is replaced with the ratio consistent variance estimator given in Theorem 3.2 of \citet*{ayyala2017mean}. For asymptotic normality
using ratio consistent estimator, uncorrelatedness of $B_{ij}$s is sufficient and hence  two main 
statements (Theorem 3.1 and 3.2 of their paper) of  \citet*{ayyala2017mean} remain valid and also the numerical study 
presented in \citet{ayyala2017mean} 
is meaningful.  Second, for asymptotic power calculation, \citet*{ayyala2017mean} assert that asymptotic normality of the statistic 
$\frac{M_{n} - \mu^{\rm T} \mu}{\sqrt{\widehat{{\rm Var}\left(M_{n}\right)}}}$ follows from ${\rm E}\left(M_{n}\right) = \mu^{\rm T}\mu$. 
However, in order to have this result, the local alternative condition needs to be modified. The power function holds only if the mean belongs to the local alternative mentioned in \eqref{eqn:localalt2}, instead of  the condition in equation (11) in \citet*{ayyala2017mean}. 
To show that the statistic is asymptotically normal under the 
alternative hypothesis,
more complicate proof is required. 
Finally, we take this opportunity to list some other typos in \citet{ayyala2017mean}:

\begin{enumerate}

\item On pp. 138, the identity $\mathrm{E}\left[n\bar{X}_{n}^{\top}\bar{X}_{n} - {\rm tr}(S_{n})\right] = \mu^{\top} \mu$ should be  ${\rm E}\left[ \bar{X}_{n}^{\top}\bar{X}_{n} - n^{-1}{\rm tr}(S_{n}) \right] = \mu^{\top} \mu$.

\item On pp. 139, the expected value of Euclidean norm of sample mean should be 
\begin{equation} \nonumber
\mathrm{E}\left[\bar{X}_{n}^{\rm T}\bar{X}_{n}) \right] = \mu^{\top} \mu+\frac{1}{n^{2}}\sum_{h \in \mathcal{M}}\sum_{k \in \mathcal{M}} {\rm tr}\left\{\Gamma(h-k)\right\} = \mu^{\rm T} \mu + \frac{1}{n}{\rm tr}\left(\Omega_{n}\right).
\end{equation}

\item The value of $b_n(0)$ is mentioned on pp. 139 do not have a factor of $1/n$ multiplied. It should be changed by mutiplying by $n$, i.e. $b_n(0) = 1$ and $b_n(i) = 2\{1 - (i - 1)/M\}$ for $i \in \{2, \ldots, M + 1 \}$. The expected value of $M_{n} = \bar{X}_{n}^{\top}\bar{X}_{n} - \frac{1}{n}\widehat{{\rm tr}(\Omega_{n})}$ is $\mu^{\top} \mu$.

\item The value of $\pi_{n}$ presented on pp. 139 should be corrected as
\begin{align*}
\pi_{n}(t,s) & = \frac{1}{n^{2}} \left\{1 - \frac{1}{n}\sum_{h \in \mathcal{M}_{+}}\left(1-\frac{h}{n}\right)\beta_{n}(h)\right\} \\
& - \sum_{h \in \mathcal{M}_{+}} \left\{ \frac{\beta_{n}(h)}{n}I(t+h=s) - \frac{\beta_{n}(h)}{n^{2}}\left( I(t \le n-h) + I(t > h) \right) \right\}.
\end{align*}

\end{enumerate}

\bibliographystyle{apa}
\bibliography{ref}

\end{document}